\RequirePackage{fix-cm}
\documentclass[smallextended]{svjour3}       
\usepackage{amssymb,graphics,calrsfs}
\usepackage{amsmath,enumerate}

\usepackage{geometry}        
\usepackage[parfill]{parskip}  
\usepackage{graphicx}
\usepackage{textcomp} 

\usepackage[applemac]{inputenc}
\usepackage{ae, aeguill}

\usepackage{pdfsync}
\begin{document}

\title{ON THE CONTINUITY OF LYAPUNOV EXPONENTS OF RANDOM WALKS IN RANDOM POTENTIALS 
}


\author{Thi Thu Hien LE        
}


\institute{T.T.Hien LE \at
              Universit\'e de Brest, UMR CNRS 6205, 29238 Brest cedex, France \\
              \email{lethithuhiensp@gmail.com}                       
}

\date{Received: date / Accepted: date}

\titlerunning{Continuity of Lyapunov exponents}
\maketitle

\begin{abstract}
We consider a simple random walk in an i.i.d. non-negative potential on the d-dimensional integer lattice $\mathbb{Z}^d$, $d \geq 3$. We study Lyapunov exponents, and present a probabilistic proof of its continuity when the potentials converge in distribution.
\end{abstract}

\section{Introduction}
\label{intro}
Let $S_n, n \in \mathbb{N}$ be the simple random walk on $\mathbb{Z}^d$, $d \geq 3$. We denote by $P_x$ and $E_x$ the probability measure and the expectation, respectively, of the random walk starting from position $x$. Independently of the random walk, we give ourselves a family of non-negative random variables $V(x,\omega), x\in\mathbb{Z}^d$ that we call the potentials. We suppose that the potentials are independent and identically distributed, with distribution function $F$, defined on a probability space $(\Omega, \mathcal{F}, \mathbb{P})$ (and associated expectation $\mathbb{E}$).

We denote by $\mathcal{D}$ the set of distribution functions $F$ which assign zero measure to the half-line $]-\infty,0[$ and $F(0)<1$. And $\mathcal{D}_1$ denotes the subset of $\mathcal{D}$ which contains all distribution functions of finite mean. For $y\in\mathbb{Z}^d$, let us write $H(y)$ for the hitting time of the walk at site $y$:
\begin{equation}\label{equa 1.1}
H(y):=\inf\{ n\geq 0: S_n=y\},
\end{equation}
with the convention that $\inf\emptyset=+\infty$. Let $V$ be a potential with distribution function $F \in \mathcal{D}$. For any $x,y \in \mathbb{Z}^d$, $\omega \in \Omega$ we define:
\begin{equation}\label{equa 1.2}
 e(x,y,\omega):=E_{x}[\exp(-\sum_{m=0}^{H(y)-1}V(S_m,\omega)), H(y)<\infty] \qquad (e(x,y,\omega)=1 \mbox{ if } x=y)
\end{equation}
Consider a Markov chain on the extended state space $\mathbb{Z}^d\cup\{\bigtriangleup\}$ where $\bigtriangleup$ is an absorbing state. At each step, the walk jumps to $\bigtriangleup$ from $x$ with probability $1-e^{-V(x)}$. Otherwise, it behaves as a simple symmetric random walk on $\mathbb{Z}^d$. The path measure of this random walk starting at $x$ in a fixed potential $V(x,\omega)$ will be denoted by $\breve{P}_{x,\omega}$. One can think of $e(x,y,\omega)$ as the probability that the random walk reaches $y$ before being killed: $e(x,y,\omega)=\breve{P}_{x,\omega}(H(y)<\infty)$. Let us now introduce for $\omega\in\Omega$ and $x,y\in \mathbb{Z}^d$ the quenched path measure:
\begin{equation}\label{equa 1.5}
\hat{P}_{x,\omega}^{y}(\cdot):=\breve{P}_{x,\omega}(\cdot \mid H(y)<\infty),
\end{equation}
and the annealed path measure:
\begin{equation}\label{equab 1}
\mathbb{\hat{P}}_x^y(\cdot):=\breve{P}_{x}(\cdot \mid H(y)<\infty) \mbox{ where } \breve{P}_{x}(\cdot)=\mathbb{E}\breve{P}_{x,\omega}(\cdot).
\end{equation}
The expectation with respect to $\hat{P}_{x,\omega}^{y}$ and $\hat{\mathbb{P}}_{x}^{y}$ are denoted by $\hat{E}_{x,\omega}^{y}$ and $\hat{\mathbb{E}}_{x}^{y}$, respectively.

Let us define:
\begin{equation}\label{equa 1.3}
 a(x,y,\omega):=-\ln e(x, y, \omega)\in [0,\infty[,
\end{equation}
and
\begin{equation} \label{equab 2}
b(x,y):=-\ln \mathbb{E}(e(x,y,\omega)).
\end{equation}
When we want to emphasize the law of potential, we write $a(x,y,F,\omega)$, $b(x,y,F)$, $\hat{\mathbb{E}}_{x,F}^{y}$ instead of $a(x,y,\omega)$, $b(x,y)$, $\hat{\mathbb{E}}_ {x}^{y}$, respectively. The quantity $a(x, y,\omega)$ can be interpreted as the weighted average over all the paths from $x$
to $y$ of the random walk in the potential $V$. The following result is contained in \cite{Zer98}:

\textbf{Theorem A}. Let $F \in \mathcal{D}_1$. There is a non-random norm $\alpha_F(x)$ on $\mathbb{R}^d$, such that $\mathbb{P}$-a.s and in $L^1(\mathbb{P})$, for all $x\in \mathbb{Z}^d$:
\begin{equation}\label{equa 1.4}
\lim_{n\to \infty}\frac{1}{n}a(0,nx,\omega)=\lim_{n\to \infty}\frac{1}{n}\mathbb{E}[a(0,nx,\omega)]=\inf_{n \in \mathbb{N}}\frac{1}{n}\mathbb{E}[a(0,nx,\omega)]=\alpha_F(x).
\end{equation}
The norm $\alpha_{F}$ is called the quenched Lyapunov exponent. Moreover, $\alpha_F$ is monotone with respect to the potential: if $ F_1,F_2 \in \mathcal{D}_1 \mbox{ and } F_1\geq F_2$ (that is $F_1(t)\geq F_2(t)$ for all $t\in \mathbb{R}$), then $\alpha_{F_1}\leq \alpha_{F_2}$.

Flury \cite{Flury06} proved:

\textbf{Theorem B}. Let $F\in\mathcal{D}$. There is a non-random norm $\beta_F(x)$ on $\mathbb{R}^d$, such that for all $x\in \mathbb{Z}^d$:
\begin{equation}\label{equation 2}
\lim_{n\to \infty}\frac{1}{n}b(0,nx)=\inf_{n\in\mathbb{N}}\frac{1}{n}b(0,nx)=\beta_F(x).
\end{equation}
The norm $\beta_{F}$ is called the annealed Lyapunov exponent. $\beta_F$ is monotone with respect to the potential: that is  if $F_1 \geq F_2,$ then $\beta_{F_1}\leq \beta_{F_2}$. The norm $\beta_{F}$ inherits from $b(0,x)$ the following upper and lower bounds:
\begin{equation}\label{equation 3}
-\ln \int e^{-t}dF(t)\leq \frac{b(0,x)}{|x|}\leq \ln 2d - \ln \int e^{-t}dF(t).
\end{equation}
About the relation between these two Lyapunov exponents, we have by Jensen's equality: $\beta_F \leq \alpha_F$. Moreover, it was showed by Zygouras \cite{Zygr09} that for every $\lambda >0$ there is $\gamma^*(\lambda)>0$ such that for all $\gamma \in ]0,\gamma^*(\lambda)[$ : $\alpha_F=\beta_F$ (where $F$ is the distribution function of the potential $\lambda+\gamma V$).

Theorems A and B are analogous to the existence of the time constant in first passage percolation. The analogy between first passage percolation and Brownian motion in Poissonian potential was first described by Sznitman \cite{Sznit94}. In particular, he proved an analogue of the shape theorem of Cox and Durrett \cite{Durret}. Zerner showed some relations between quenched Lyapunov exponent of random walk in random potential and first passage percolation (see Proposition 9 in \cite{Zer98}). Recently, Sodin \cite{Sodin13} proved two theorems on concentration inequalities for random walk in random potential which are counterparts of Talagrand \cite{Talagrand95} and Benjamini-Kalai-Schramm \cite{Benja03}.

We want to study in this article the continuity of Lyapunov exponents with respect to the law of the potential. For random walk in random potential as a model of random polymers, when the potential is a function of an ergodic environment and steps of the walk, lemma 3.1 of \cite{Rass-Sep12} has showed the $L^p$ continuity ($p>d$) of the quenched point-to-point free energy with respect to the law of potential. For first-passage percolation, Cox\cite{Cox81} proved the continuity of the time constant with respect to the law of the passage time. Scholler\cite{Scholler11} also studied this question for a random coloring model which is a dependent first passage percolation model.

In our context, this problem is mentioned in section $11$ of Mourrat\cite{Mour12}. Here our proof is based on the argument of \cite{Cox81}. The most difficult part is to show that $\liminf \alpha_{F_n}\geq \alpha_F$ and $\liminf \beta_{F_n}\geq \beta_F$  if $F_n\stackrel{w}{\to}F$. It brings us to the questions of the ballisticity of the random walk under the conditional quenched and annealed path measure. By adapting the techniques used by Sznitman \cite{Sznit95} for the Brownian motion in a Poissonian potential we obtain a similar result for the quenched path measure. Note however that the boundedness condition on $W$ is not needed in the discrete case. And theorem 1.1 in \cite{KM12} showed that the velocity under annealed path measure of the walk that reaches $y$ is bounded above as $|y| \to \infty$. Independently of \cite{KM12}, \cite{Ioffe-Velenik} used a different method but it also implies the same conclusion (see Theorem C in \cite{Ioffe-Velenik}). But none of these two papers gives an explicit expression for the constant. When $d\ge 3$, by a simpler argument, we can control this constant in order to prove the continuity of the annealed Lyapunov exponent with respect to the law of potential (part (ii) of Theorem \ref{theorem 1.1}). We now state our main results.
\begin{theorem}\label{theorem 1.3}
Assume that $(F_n)$ is a sequence of distribution functions such that $F_n\in \mathcal{D}$ and there is a distribution function $G \in \mathcal{D}_1$, $G \leq F_n$ for all $n$. Furthermore, suppose that there exists a constant $\lambda>0$ satisfying $F_n(\lambda)=0$ for all $n$ if $d=2$. Then $\lim_{n\to\infty}\alpha_{F_n}(x)=\alpha_{F}(x) $ for all $x\in\mathbb{Z}^d$ if $F_n\stackrel{w}{\to}F$.
\end{theorem}
\begin{theorem}\label{theoremb 1}
Assume that $(F_n)$ is a sequence of distribution functions such that $F_n \in \mathcal{D}$, $F_n\stackrel{w}{\to}F$, $F\in \mathcal{D}$ and $F$ assigns probability $1$ to $[0,+\infty[$. Furthermore, suppose that there exists a constant $\lambda>0$ satisfying $F_n(\lambda)=0$ for all $n$ if $d=2$. Then $\lim_{n\to\infty}\beta_{F_n}(x)=\beta_{F}(x) $ for all $x\in\mathbb{Z}^d$.
\end{theorem}

\begin{remark}
The condition involving $G$ in the theorem \ref{theorem 1.3} ensures that all distribution functions have finite means. But in the case of annealed Lyapunov exponent, $\beta_{F}$ exists even when $\int tdF(t)=+\infty$. This is the reason that we don't require that the means are finite in Theorem \ref{theoremb 1}.
\end{remark}
\begin{remark}
The additional condition when $d=2$ in boths theorems above means that all potentials are bounded below. It is useful for our technique. But we do not think that this hypothesis is necessary for the continuity of Lyapunov exponents.  
\end{remark}
The proof of Theorem \ref{theorem 1.3} is divided in two parts. The first step is to prove that $\limsup\alpha_{F_n}(x)\leq \alpha_{F}(x)$ and then the proof of $\liminf\alpha_{F_n}(x)\geq \alpha_{F}(x)$. It is surprising that the proof of $\limsup$ is relatively easy, while the proof of $\liminf$ is more difficult. We state below in great generality the main results of Cox\cite{Cox81}. It combines the results of Proposition 4.4, lemma 4.7 and proof of theorem 1.14 in \cite{Cox81}. It gives a sufficient condition for $\liminf \alpha_{F_n}(x)\geq \alpha_F(x)$ to hold.

\begin{theorem}[An abstract theorem on the continuity of the time constant]\label{theorem 1.2}
Let $\mu$: $\mathcal{D}_1\longrightarrow \mathbb{R}^+, F\mapsto \mu(F)$ be a map that satisfies the following three conditions:
\begin{itemize}
\item[\rm{(i)}]$\mu(F)\leq \mu(G)$ if $F\geq G$.
\item[\rm{(ii)}]For all $F \in \mathcal{D}_1$, there exists $c_1(F)>0$ and $f_1(F)>0$ such that:
\begin{itemize}
\item[\rm{(1)}] $c_1(F)\leq c_1(G)$ if $F\geq G$,
\item[\rm{(2)}] $\lim_{n\to\infty}f_1(F_n)=f_1(F)$ for all $F_n \in \mathbb{F} $ such that $F_n\stackrel{w}{\to}F$,
\item[\rm{(3)}] $\mu(F*G)\leq \mu(F)+c_1(F)f_1(F)\int tdG(t)$ for all $G \in \mathbb{F}$.
\end{itemize} 
\item[\rm{(iii)}]For all $F\in\mathcal{D}_1$, $t_0>0$ there exists $c_2(F)>0$ and $f_2(t_0,F)$ such that  :
\begin{itemize}
\item[\rm{(1)}] $c_2(F)\leq c_2(G)$ if $F\geq G$,
\item[\rm{(2)}] $\lim_{n\to\infty}f_2(t_0,F_n)=0$ for all $F_n \in \mathbb{F} $ such that $\lim_{n \to \infty}F_n(t_0-)=0$,
\item[\rm{(3)}] $|\mu(F^{t_0})-\mu(F)|\leq c_2(F)f_2(t_0,F)$ where:
\begin{equation}\label{equa 1.10}
F^{t_0}(t):=
\begin{cases}
0 \mbox{ if $t<t_0$}\\
F(t) \mbox{ if $t\geq t_0$}
\end{cases}
\end{equation}
\end{itemize}
\end{itemize}
Then $\liminf_{n \to \infty}\mu(F_n)\geq \mu(F)$ if $F_n\stackrel{w}{\to}F$.
\end{theorem}
Condition $(i)$ is the monotonicity property of Lyapunov exponent refered to in theorem A. It is a key tool for proving our continuity results. Because of this property, when dealing with $F_n\stackrel{w}{\to}F$, it suffices to consider only two cases: $F_n\leq F$ for all $n$ and $F_n\geq F$ for all $n$. To see this, define $\underline{F}_n(t)=\min\{F_n(t), F(t)\}$ and $\overline{F}_n(t)=\max\{F_n(t), F(t)\}$, so that $\underline{F}_n\leq F_n(t)\leq \overline{F}_n(t)$. Then $\alpha_{\overline{F}_n}\leq \alpha_{F_n}\leq \alpha_{\underline{F}_n}$, $\alpha_{\overline{F}_n}\leq \alpha_{F}\leq \alpha_{\underline{F}_n}$ and both $\underline{F}_n, \overline{F}_n \stackrel{w}{\to}F$ whenever $F_n\stackrel{w}{\to}F$.

Condition $(ii)$ will be proved in Proposition \ref{propo 3.2} and Condition $(iii)$ will follow from Corollary \ref{corollary 3.1}.

As in \cite{Cox-Kes81}, the proof of Theorem \ref{theoremb 1} is done in two steps. First, we will show the continuity of $\beta_{F}$ under the same hypothesis of theorem \ref{theorem 1.3} and the proof of it is totally similar to the case $\alpha_{F}$. Next, to eliminate the condition of finite mean, with $t_0>0$ arbitrary, in the theorem \ref{theoremb 3} of this paper, we prove that $\beta_{^{t_0}F} \to \beta_{F}$ when $t_0 \to \infty$, where $^{t_0}F$ is the distribution function obtained by truncating below at $t_0$ (see (\ref{equa 4.3}) for the definiton of $^{t_0}F$).

For $y=(y_1,y_2,...,y_d)\in \mathbb{Z}^d $, $|y|$ denotes the $\ell_1$-norm of $y$: $|y|=|y_1|+|y_2|+\cdots |y_d|$ while $\|y\|$ designates the $\ell_{\infty}$ of $y$: $\|y\|=\max_{1\leq i\leq d}|y_i|$. And $|A|$ is the cardinality of the set $A$.
\begin{theorem}\label{theorem 1.1}
Let $d\geq 3$ and $V$ be a potential with distribution function $F$,
\begin{itemize}
\item[\rm{(i)}]If $F\in\mathcal{D}_1$, there is a set $\bar{\Omega}$ of full $\mathbb{P}$ probability and a constant $\kappa(d, F)\in ]0, \infty[$ such that for all $\omega\in \bar{\Omega}$, 
\begin{equation}\label{equa 1.6}
\limsup_{|y|\to\infty}\frac{\hat{E}_{0,\omega}^{y}(H(y))}{|y|} <\kappa.
\end{equation}
\item[\rm{(ii)}]If $F\in\mathcal{D}$, there exists a constant $D(d)$ such that for all $y\in \mathbb{Z}^d, y\neq 0$:
\begin{equation}
\frac{\hat{\mathbb{E}}_{0}^{y} (H(y))}{|y|}\leq D(d) \frac{1}{-\ln \int \exp(-t)dF(t)}\big(\ln 2d -\ln \int \exp(-t)dF(t)\big).
\end{equation}
\end{itemize} 
\end{theorem}

\section{Proof of theorem \ref{theorem 1.1}}
\textbf{Proof of part (i) of the theorem \ref{theorem 1.1}.}

Brownian motion in a Poissonian potential under the quenched law was treated in \cite{Sznit95}. We show how to adapt these arguments to prove part (i) of theorem \ref{theorem 1.1}.\\
Let $d\geq 3$ and $y\in \mathbb{Z}^d$ be the ''target point''. For $A\subset \mathbb{Z}^d$, $H(A)$ denotes the entrance time of $S_n$ in $A$: $H(A):=\inf\{n\geq 0: S_n \in A\}$. By convention $\inf \emptyset=+\infty$.\\
Choose some $\delta=\delta(F) >0$ such that $\rho=\mathbb{P}(V(0)\geq \delta)>0$. We fix a large even integer $l=l(d,F)$. We will explain how to choose $l$ after equation (\ref{equa 2.17}).\\ 
Let us now introduce a partition of $\mathbb{Z}^d$, namely $\{C(q), q \in \mathbb{Z}^d\}$ where $C(q)$ is the cube of side length $l$ and center $lq$:
\begin{equation}\label{equa 2.1}
C(q)=(lq+[-l/2,l/2)^d)\cap \mathbb{Z}^d.
\end{equation}
Let us recall the definitions of an ''occupied set'' and an ''empty set'' used in \cite{KM12}. Given an environment $\omega\in\Omega$ and a target point $y \in \mathbb{Z}^d$, we say that a set $A\subset \mathbb{Z}^d$ is occupied if there exists $z\in A, z\neq y$ such that $V(z,\omega)\geq \delta$. We say that $A$ is empty otherwise.

Define:
\begin{equation}\label{equa 2.2}
\mathcal{C}_1=\{ q \in \mathbb{Z}^d: C(q)\mbox{ is occupied } \}.
\end{equation}
\begin{equation} \label{equa 2.3}
\mathcal{C}_2=\{ q\in \mathbb{Z}^d: C(q) \mbox { is empty}\}.
\end{equation}
Accordingly, we define:
\begin{equation}\label{equa 2.11}
H_i:=\sum_{q\in\mathcal{C}_i}\sum_{m=0}^{H(y)-1}\mathbf{1}_{\{S_m \in C(q)\}} \hspace{0.1cm},\hspace{0.5cm} i=1,2.
\end{equation}
That is, $H_i$ is the time spent by process in boxes indexed by the class $\mathcal{C}_i$ until it reaches $y$.

Now we shall estimate $H_1$ by the following lemma, which is totally like lemma $2.1$ in \cite{KM12}.
\begin{lemma}\label{lemme 2.1}
There exists a constant $C_1=C_1(d,F)$ and a set $\Omega_1$ of full $\mathbb{P}$-measure such that for all $\omega \in \Omega_1$ and for all $y\in \mathbb{Z}^d\backslash \{0\}$,
$$\hat{E}_{0,\omega}^{y}(H_1)\leq C_1|y|. $$
\end{lemma}
\begin{proof}
We use  the same argument as in lemma $2.1$ of \cite{KM12}. Then we don't repeat here. But we remark that with:
$$\sigma_1:=\inf\{n\geq 0: S_n \in \cup_{q\in \mathcal{C}_1}C(q)\}, \hspace*{0.5cm} \sigma_{m+1}:=\inf\{n\geq \sigma_m+ dl: S_n \in \cup_{q\in \mathcal{C}_1}C(q)\},$$
$$Y:=\big|\{ m: \exists n \in [\sigma_m, \sigma_m+dl] \mbox{ such that } V(S_n)\geq \delta\}\big|,$$
and $\epsilon \in (0,1)$ arbitrary, we have by the Markov property:
\begin{eqnarray} 
\notag 
\breve{P}_{0,\omega}(\sigma_m<H(y)<\infty)&= \breve{P}_{0,\omega}(\sigma_m<H(y)<\infty, Y \geq (m-1)\epsilon) \\ \notag 
&\hspace{0.4cm}+\breve{P}_{0,\omega}(\sigma_m<H(y)<\infty,Y < (m-1)\epsilon) \\ 
&\leq \exp(-(m-1)\epsilon \delta)+ P(Z<(m-1)\epsilon), 
\end{eqnarray}
where $Z$ is a binomial random variable with parameters $(m-1)$ and $(2d)^{-dl}$. Note that we use here: $P(\exists n \in [\sigma_m,\sigma_m+dl], n< H(y): V(S_n)\geq \delta)\geq (2d)^{-dl}$. Moreover, we can find $\alpha_0>0$ such that $e(0,y,\omega)\geq \exp(-\alpha_0|y|)$ since $\lim_{n\to\infty}\frac{-\ln e(0,ny,\omega)}{n}=\alpha(y)>0$ (Theorem A). Then we are now ready to follow the argument of lemma $2.1$ in \cite{KM12}.
\end{proof}
We shall now estimate the total number of cubes visited by the random walk $S_m$ up to time $H(y)$. The argument we follow is very similar to \cite{Sznit95}. We define:
\begin{equation}\label{equa 2.4}
\mathcal{A}_1=\{q \in \mathbb{Z}^d: H(C(q))<H(y)\}.
\end{equation}
Now $\hat{P}_{0,\omega}^y$-a.s. $H(y)$ is finite and, therefore,
\begin{equation} \label{equa 2.5}
\hat{P}_{0,\omega}^y\mbox{-a.s. }\mathcal{A}_1 \mbox{ is a lattice animal (i.e, a finite connected set) of $\mathbb{Z}^d$ containing $0$},
\end{equation}
where we use the standard adjacency relation for which $q,q'$ are adjacent if $|q-q'|\leq 1$. Let us now explain how to choose the side length $l$ of the cubes. We consider $\mathbb{P}$-probability that there exists an lattice animal $\Gamma$ containing $0$, of size $n$ and such that: 
\begin{equation}\label{equa 2.6}
\sum_{q\in \Gamma}\mathbf{1}_{\{C(q,l/4) \mbox{ is occupied }\}}\leq n/2,
\end{equation}
where $C(q,l/4)$ denotes the cube of side length $l/4$ and center $lq$. We first note that there are less than $(2d)^{2n}$ distinct animals $\Gamma$ in $\mathbb{Z}^d$ of cardinality $n$ containing $lq$. To see this, we use a spanning tree of $\Gamma$ with $n$ vertices and $n-1$ nearest neighbor edges. We have ''the number of lattice animals of cardinality $n$ and contain $lq$'' is less than ''the number of nearest neighbor paths starting at $lq$ of length at most $2n$ walking around the spanning tree''. With the definition of adjacent vertices above, the number of these paths does not exceed $(2d)^{2n}$ (see also \cite{Harry93}). By the definition of an occupied cube, we have: $p=\mathbb{P}( C(q,l/4) \mbox{ is occupied })=1-\mathbb{P}(C(q,l/4) \mbox{ is empty })=1-(1-\rho)^{(l^d/4^d)}$. Then,
\begin{align}
\mathbb{P}\big[\exists  \mbox{ an lattice animal } &\Gamma: 0 \in \Gamma ,\vert\Gamma\vert=n,\notag \\
&\sum_{q\in \Gamma}\mathbf{1}_{\{C(q,l/4) \mbox{ is occupied }\}}\leq n/2\big] \leq (2d)^{2n}p_n(l,d)\notag, 
\end{align}
where $p_n(l,d)$ stands for the probability that a binomial variable with parameters $n$ and $p=1-(1-\rho)^{(l^d/4^d)}$ takes a value smaller than $n/2$. Note that if $l$ is large, $p$ is close to $1$. We choose $l$ large enough such that:
\begin{equation}\label{equa 2.17}
\sum_{n=1}^{\infty}(2d)^{2n}p_n(l,d)<\infty,
\end{equation}
That such a choice of $l$ is possible follows from standard exponential estimates on the binomial distribution with success probability $p$ close to $1$. 

By Borel-Cantelli lemma, there is a set $\Omega_2$ of full $\mathbb{P}$-measure such that:
\begin{align}\label{equa 2.8}
\mbox{ for all }& \omega\in\Omega_2, \mbox{ there is } n_0(\omega) \mbox{ so that for all } n\geq n_0(\omega) \mbox{ and}\notag \\
&  \Gamma \mbox{ a } \mbox{ lattice animal containing $0$, with $\vert \Gamma\vert=n,$}\notag\\
 &\sum_{q\in\Gamma}\mathbf{1}_{\{C(q,l/4)\mbox{ is occupied }\}} \geq n/2.
\end{align}
Let us explain the strategy underlying the proof of an exponential estimate under $\hat{P}_{0,\omega}^{y}$ on the size of $\mathcal{A}_1$. The idea is to use (\ref{equa 2.8}), so that for typical configurations $\omega$ and large $\mathcal{A}_1$, the number of occupied sites in $\mathcal{A}_1$ represents a nonvanishing fraction of $\vert\mathcal{A}_1\vert$. 
Lemma \ref{lemme 2.2} below plays the role of theorem 1.3 of \cite{Sznit95}. We don't present its demonstration in detail.
\begin{lemma}\label{lemme 2.2}
There exists a set $\Omega_2$ of full $\mathbb{P}$ measure and $C_2(d,F) > 0$ such that, for $\omega\in \Omega_2$,
\begin{equation}\label{equa 2.9}
\sup_y\Big(e(0,y,\omega)\hat{E}_{0,\omega}^{y}\big[\exp(C_2\vert\mathcal{A}_1\vert)\big]\Big)<\infty.
\end{equation}
\end{lemma}
The constant $C_2$ is chosen as: 
\begin{equation}\label{equa 2.19}
C_2=\frac{1}{2\cdot 3^d}\ln (1/\chi)>0,
\end{equation} 
where: $$\chi=\sup_{||z||\leq l/2, C(0,l/4)\mbox{ is occupied}}E_z(\exp(-\sum_{m=0}^{T_1-1}V(S_m))),$$
and $T_1$ is the time of travel of $S$ at $\parallel\cdot\parallel$ distance $3l/4$ defined by:
$$T_0=0, \hspace{0.5cm} T_1=\inf \lbrace m\geq 0, ||S_m-S_0||\geq 3l/4\rbrace. $$
We want to explain why $\chi$ chosen above is strictly smaller than $1$ and hence $C_2>0$. If $C(0,l/4)$ is occupied, there exists $z_0\neq y, z_0 \in C(0,l/4)$ such that $V(z_0,\omega)\geq \delta$. When $||z||\leq l/2$ and $C(0,l/4)$ is occupied, by the strong Markov property, we have:
\begin{align}\label{equa 2.10}
E_z\big( \exp(-\sum_{m=0}^{T_1-1}V(S_m))\big)&=\breve{P}_{z,\omega}(T_1 < \infty)\notag \\
&=\breve{P}_{z,\omega}(T_1<H(z_0))+\breve{P}_{z,\omega}( H(z_0)\leq T_1<\infty)\notag \\
&=\breve{P}_{z,\omega}(T_1< H(z_0))+\breve{P}_{z,\omega}(H(z_0)\leq T_1)\breve{P}_{z_0,\omega}(T_1 < \infty)\notag \\
&\leq P_{z}(T_1< H(z_0))+P_{z}(H(z_0)\leq T_1)\exp(-\delta)\notag\\
&\leq 1+(e^{-\delta}-1)P_{z}(H(z_0)\leq T_1)\notag\\
&\leq 1+(e^{-\delta}-1)(2d)^{-\frac{-3ld}{2}} < 1,
\end{align}
where we now use $l/2+l/8<3l/4$ so that $z_0$ is strictly within $\|\cdot \|$ distance $3l/4$ from $z$ and consequently $P_z(H(z_0)\leq T_1)>(2d)^{-\frac{-3ld}{2}}$.\\
By using (\ref{equa 2.8}) and the choice of $C_2$ in (\ref{equa 2.19}), it is now easy to follow the argument as in theorem 1.3 of \cite{Sznit95} to obtain (\ref{equa 2.9}).
\begin{lemma}\label{lemme 2.3}
There exists a constant $C_3=C_3(d,F)<\infty$ and a set $\Omega_2$ of full $\mathbb{P}$ measure such that for all $\omega \in \Omega_2$:
\begin{equation}\label{equa 2.16}
\limsup_{|y| \to \infty}\frac{\hat{E}_{0,\omega}^{y}(H_2(y))}{|y|} <C_3.
\end{equation}
\end{lemma}
\begin{proof}
Using the strong Markov property, we have:
\begin{align}\label{equa 2.13}
 \hat{E}_{0,\omega}^{y}(H_2(y))&=\sum_{q\in \mathcal{C}_2}\hat{E}_{0,\omega}^{y}\big(\sum_{m=0}^{H(y)-1}\mathbf{1}_{\{S_m\in C(q)\}}\big)\notag\\
 &=\sum_{q\in \mathcal{C}_2}\hat{E}_{0,\omega}^{y}\Big(\hat{E}_{0,\omega}^{y}\big(\sum_{m=0}^{H(y)-1}\mathbf{1}_{\{S_m\in C(q)\}}\vert\mathcal{F}(H(C(q)))\big) \Big)\notag\\
 &=\sum_{q\in \mathcal{C}_2}\hat{E}_{0,\omega}^{y}\Big(H(C(q))<H(y),\hat{E}_{S_{H(C(q))},\omega}^y\big(\sum_{m=0}^{H(y)-1}\mathbf{1}_{\{S_m\in C(q)\}} \big)\Big).
 \end{align}
For $z\in C(q)$, we consider:
\begin{align}\label{equa 2.14}
\hat{E}_{z,\omega}^y&\big(\sum_{m=0}^{H(y)-1}\mathbf{1}_{\{S_m\in C(q)\}} \big)\notag\\
&=\sum_{z'\in C(q)}\hat{E}_{z,\omega}^y\big(\sum_{m=0}^{H(y)-1}\mathbf{1}_{\{S_m=z'\}} \big)\notag\\
&=\frac{1}{e(z,y,\omega)}\sum_{z'\in C(q)}E_{z}\bigg(\sum_{m=0}^{H(y)-1}\mathbf{1}_{\{S_m=z'\}}\exp\big(-\sum_{m=0}^{H(y)-1}V(S_m)\big),H(y)<\infty\bigg)\notag\\
&=\frac{1}{e(z,y,\omega)}\sum_{z'\in C(q)}\sum_{k=0}^{+\infty}E_{z}\bigg(\mathbf{1}_{\{S_k=z'\}}\exp\big(-\sum_{m=0}^{H(y)-1}V(S_m)\big),k<H(y)<\infty\bigg)\notag\\
&=\frac{1}{e(z,y,\omega)}\sum_{z'\in C(q)}\sum_{k=0}^{+\infty}E_{z}\big(\mathbf{1}_{\{S_k=z'\}} \exp(-\sum_{m=0}^{k-1}V(S_m)),k<H(y)\big)\breve{P}_{z',\omega}(H(y)<\infty)\notag\\
&\leq \sum_{z'\in C(q)}\frac{e(z',y,\omega)}{e(z,y,\omega)}\sum_{k=0}^{+\infty}E_z\big(\mathbf{1}_{\{S_k=z'\}}\big).
\end{align}
If $z,z'\in C(q)$ and $C(q)$ is empty, then:
\begin{align}\label{equa 2.15}
&e(z,z',\omega)\notag\\
&=E_z\big(\exp(-\sum_{m=0}^{H(z')-1}V(S_m)), H(z')< \infty\big)\notag\\
&\geq E_z\big(\exp(-\sum_{m=0}^{H(z')-1}V(S_m)), \mbox{ $S_m$ follows a path from $z$ to $z'$ in $C(q)$ of length $\leq$ $dl$}\big)\notag\\
&\geq \exp(-dl\delta)(2d)^{-dl}=const.
\end{align}
By proposition $2$ in \cite{Zer98}, for all $z,z',y \in \mathbb{Z}^d$, we have: $\frac{e(z',y,\omega)}{e(z,y,\omega)}\leq \frac{1}{e(z,z',\omega)}$. If $C(q)$ is empty, from (\ref{equa 2.15}), we can find a constant $K(d,F)$ such that: 
\begin{equation}\label{equa 2.18}
\frac{e(z',y,\omega)}{e(z,y,\omega)}\leq K(d,F) \qquad  \mbox{for all } z,z'\in C(q),y \in \mathbb{Z}^d.
\end{equation}
From (\ref{equa 2.13}), (\ref{equa 2.14}) and (\ref{equa 2.18}), we obtain:
$$\hat{E}_{0,\omega}^{y}(H_2(y))\leq \sum_{q\in\mathcal{C}_2}K(d,F)\hat{P}_{0,\omega}^{y}[H(C(q))<H(y)]\sup_{z\in C(q)}\sum_{z'\in C(q)}\sum_{k=0}^{+\infty}E_z\big(\mathbf{1}_{\{S_k=z'\}}\big). $$
For $d\geq 3$, the simple walk  is transient, $\sup_{z\in C(q)}\sum_{z'\in C(q)}\sum_{k=0}^{+\infty}E_z\big(\mathbf{1}_{\{S_k=z'\}}\big)=const(d)<\infty$, we see that:
$$\hat{E}_{0,\omega}^{y}[H_2(y)]\leq C_4(d,F)\hat{E}_{0,\omega}^{y}[|\mathcal{A}_1|],$$
where the definition of $\mathcal{A_1}$ is given by (\ref{equa 2.4}). By lemma \ref{lemme 2.2}, there exists a set $\Omega_2$ of full $\mathbb{P}$ measure and $C_5(\omega)\in (0,\infty)$ which depends only on $\omega$ such that for all $\omega \in \Omega_2$:
\begin{align}
\limsup_{|y|\to\infty}\Big(\frac{\hat{E}_{0,\omega}^{y}C_2|\mathcal{A}_1|}{|y|}\Big)&\leq\limsup_{|y|\to\infty}\frac{\ln \hat{E}_{0,\omega}^{y}\Big[\exp\big(C_2|\mathcal{A}_1|\big)\Big]}{|y|} \notag\\ 
& \leq \limsup_{|y|\to\infty}\frac{\ln \big(\frac{C_5(\omega)}{e(0,y,\omega)}\big)}{|y|}\notag\\
&\leq \sup _{|e|=1}\alpha_{F}(e)=const.\notag
\end{align}
We finally obtain (\ref{equa 2.16}).
\end{proof}
With the notations of (\ref{equa 2.11}), we know that:
\begin{equation}\label{equa 2.12}
\hat{E}_{0,\omega}^{y}(H(y))=\hat{E}_{0,\omega}^{y}(H_1)+\hat{E}_{0,\omega}^{y}(H_2).                          
\end{equation}
By lemma \ref{lemme 2.1} and lemma \ref{lemme 2.3}, part (i) of theorem \ref{theorem 1.1} is now proved.\\
\textbf{Proof of part (ii) of the theorem \ref{theorem 1.1}.}\\
Let $y\in\mathbb{Z}^d, y\neq 0$.\\
By the strong Markov property, we have:
\begin{align}\label{equab 3}
E_0[H(y)&\exp(-\sum_{m=0}^{H(y)-1}V(S_m)),H(y)<\infty]\notag\\
&=\sum_{z'\in\mathbb{Z}^d} E_0\Big(\sum_{m=0}^{H(y)-1}\mathbf{1}_{\{S_m=z'\}}\exp(-\sum_{m=0}^{H(y)-1}V(S_m)), H(y)<\infty \Big)\notag\\
&=\sum_{z'\in\mathbb{Z}^d} E_0\Big(H(z')<H(y),\exp(-\sum_{m=0}^{H(z')-1}V(S_m))\Big) \notag\\
&\qquad \qquad \times E_{z'}\Big(\sum_{m=0}^{H(y)-1}\mathbf{1}_{\{S_m=z'\}}\exp(-\sum_{m=0}^{H(y)-1}V(S_m)), H(y)<\infty \Big)\notag\\
&=\sum_{z'\in\mathbb{Z}^d} E_0\Big(H(z')<H(y),\exp(-\sum_{m=0}^{H(z')-1}V(S_m))\Big)e(z',y,\omega) \notag\\
&\qquad \qquad \times \frac{1}{e(z',y,\omega)} E_{z'}\Big(\sum_{m=0}^{H(y)-1}\mathbf{1}_{\{S_m=z'\}}\exp(-\sum_{m=0}^{H(y)-1}V(S_m)), H(y)<\infty \Big)\notag\\
&=\sum_{z'\in\mathbb{Z}^d} E_0\Big(H(z')<H(y),\exp(-\sum_{m=0}^{H(y)-1}V(S_m)), H(y)<\infty \Big)\notag\\
&\qquad \qquad \times \hat{E}_{z',\omega}^{y}(\sum_{m=0}^{H(y)-1}\mathbf{1}_{\{S_m=z'\}})\notag\\
\end{align}
Still by the Markov property:
\begin{align}\label{equab 4}
\hat{E}_{z',\omega}^{y}&(\sum_{m=0}^{H(y)-1}\mathbf{1}_{\{S_m=z'\}})=\frac{1}{e(z',y,\omega)}E_{z'}\bigg(\sum_{m=0}^{H(y)-1}\mathbf{1}_{\{S_m=z'\}}\exp\big(-\sum_{m=0}^{H(y)-1}V(S_m)\big),H(y)<\infty\bigg)\notag\\
&=\frac{1}{e(z',y,\omega)}\sum_{k=0}^{+\infty}E_{z'}\bigg(\mathbf{1}_{\{S_k=z'\}}\exp\big(-\sum_{m=0}^{H(y)-1}V(S_m)\big),k<H(y)<\infty\bigg)\notag\\
&=1+\frac{1}{e(z',y,\omega)}\sum_{k=1}^{+\infty}E_{z'}\big(\mathbf{1}_{\{S_k=z'\}} \exp(-\sum_{m=0}^{k-1}V(S_m)),k<H(y)\big)\breve{P}_{z',\omega}(H(y)<\infty)\notag\\
&\leq\sum_{k=0}^{+\infty}E_{z'}(\mathbf{1}_{\{S_k=z'\}}):= D(d) <\infty. \notag\\
\end{align}
since the simple random walk is transient on $\mathbb{Z}^d, d\geq 3$. We now attach to each trajectory $(S_m)_{m\geq 0}$ which starts at $0$, the lattice animal:
\begin{equation}\label{equab 5}
\mathcal{A}_2(0,y,(S_m)_{m\geq 0})=\{z\in\mathbb{Z}^d: H(z)<H(y)\}.
\end{equation}
From (\ref{equab 3}), (\ref{equab 4}) and (\ref{equab 5}):
\begin{equation}\label{equab 6}
\hat{\mathbb{E}}_{0}^{y}(H(y))=\frac{\mathbb{E}E_0[H(y)\exp(-\sum_{m=0}^{H(y)-1}V(S_m)),H(y)<\infty]}{\mathbb{E}e(0,y,\omega)}\leq D\hat{\mathbb{E}}_{0}^{y}(|\mathcal{A}_2(0,y)|).
\end{equation}
To estimate $\hat{\mathbb{E}}_{0}^{y}(|\mathcal{A}_2(0,y)|)$, we argue as in lemma 3 in \cite{Zer98}. Take $d_1:=-\ln \int e^{-t}dF(t)=-\ln \mathbb{E}(e^{-V(0)})$. By Jensen's equality and independence of $V(x)$, $x\in \mathbb{Z}^d$:
\begin{align}\label{equab 7}
d_1\hat{\mathbb{E}}_{0}^{y}(|\mathcal{A}_2(0,y)|)&\leq \ln \hat{\mathbb{E}}_{0}^{y}(\exp(d_1|\mathcal{A}_2(0,y)|))\notag\\
&\leq b(0,y)+\ln \mathbb{E} E_0[\exp(d_1|\mathcal{A}_2(0,y)|-\sum_{s\in\mathcal{A}_2(0,y)}V(s)), H(y)<\infty]\notag\\
&\leq b(0,y)+\ln E_0 \Big( \prod_{s\in \mathcal{A}_2(0,y)}\mathbb{E}(\exp(d_1-V(s))) \Big)=b(0,y)
\end{align}
From (\ref{equab 6}), (\ref{equab 7}) and (\ref{equation 3}):
\begin{align}\label{equab 11}
\frac{\hat{\mathbb{E}}_{0}^{y}(H(y))}{|y|}\leq \frac{D}{d_1} \frac{b(0,y)}{|y|}\leq \frac{D}{-\ln \int \exp(-t_1) dF(t) }\big(\ln 2d-\ln \int \exp(-t)dF(t)\big)
\end{align}
The proof of (ii) in theorem \ref{theorem 1.1} is now complete.
\section{Continuity of Lyapunov exponents}
The following proposition is a main ingredient in the proof of the continuity of Lyapunov exponent. It verifies the condition (ii) of Theorem \ref{theorem 1.2}. We learned the idea from theorem $7.12.$ in \cite{Wier78}. $F*G$ denotes the convolution of $F$ and $G$.
\begin{proposition}\label{propo 3.2}
Let $d\geq 3$. For any distribution function $F \in \mathcal{D}_1$, there exists $c_1(F)>0$ and $f_1(F)>0$ such that:
\begin{itemize}
\item[\rm{(1)}] $c_1(F)\leq c_1(G)$ for all $F, G \in \mathcal{D}_1$ such that $F\geq G$, 
\item[\rm{(2)}] $\lim_{n\to\infty}f_1(F_n)=f_1(F)$ for $F_n \in \mathcal{D}_1 $, $F_n\stackrel{w}{\to}F$,
\item[\rm{(3)}] $\alpha_{F*G}(x)\leq \alpha_{F}(x)+c_1(F)f_1(F)\int tdG(t)|x|$ for all $F, G \in \mathcal{D}_1$ and $x\in\mathbb{Z}^d$.
\end{itemize}
\end{proposition}
\begin{proof}
Let $V(x)$, $x \in \mathbb{Z}^{d}$ be i.i.d random potentials  with distribution $F$; $W(x)$, $x \in \mathbb{Z}^{d}$ be i.i.d random potentials with distribution $G$ such that the two sequences defined on a same probability space $(\Omega, \mathcal{F}, \mathbb{P})$ are independent of each other . Then, $(V+W)(x)$, $x\in\mathbb{Z}^{d}$ are i.i.d random potentials with distribution $F*G$. We have for $x\in\mathbb{Z}^d, x\neq 0, \omega \in \Omega$:
\begin{align}\label{equa 3.8}
a&(0,nx,F*G,\omega)\notag\\
&=-\ln E_0[\exp (-\sum_{m=0}^{H(nx)-1}V(S_m)-\sum_{m=0}^{H(nx)-1}W(S_m)), H(nx)<\infty] \notag \\
&=-\ln \Big[\frac{ E_0\big[\exp (-\sum_{m=0}^{H(nx)-1}V(S_m)-\sum_{m=0}^{H(nx)-1}W(S_m)), H(nx)<\infty\big]}{E_0[\exp (-\sum_{m=0}^{H(nx)-1}V(S_m)), H(nx)<\infty]}\notag\\
&\hspace{0.5cm}\cdot E_0[\exp (-\sum_{m=0}^{H(nx)-1}V(S_m)), H(nx)<\infty] \Big]\notag \\
&=-\ln \hat{E}_{0,\omega}^{nx}\Big(\exp(-\sum_{m=0}^{H(nx)-1}W(S_m))\Big)-\ln E_0[\exp (-\sum_{m=0}^{H(nx)-1}V(S_m)), H(nx)<\infty]\notag\\ 
&\mbox{(where $\hat{E}_{x,\omega}^{y}(X)=\frac{E_x[X\exp (-\sum_{m=0}^{H(y)-1}V(S_m)), H(y)<\infty]}{ E_x[\exp (-\sum_{m=0}^{H(y)-1}V(S_m)), H(y)<\infty]}$ is defined below (\ref{equab 1}))}\notag\\
&=-\ln \hat{E}_{0,\omega}^{nx}\Big(\exp(-\sum_{m=0}^{H(nx)-1}W(S_m))\Big)+a(0,nx,F,\omega)\notag \\
&\leq \hat{E}_{0,\omega}^{nx}\Big(\sum_{m=0}^{H(nx)-1}W(S_m)\Big)+a(0,nx,F,\omega).
\end{align}
Note that the last inequality is obtained by the Jensen's inequality: $\ln E(\exp (X))\geq E(X)$. We use now Fubini's theorem and the independence of $(W(x))$ and $(V(x))$:
\begin{align}\label{equa 3.9}
\mathbb{E}\hat{E}_{0,\omega}^{nx}\Big(\sum_{m=0}^{H(nx)-1}W(S_m)\Big)&=\mathbb{E}\Big(E_0\big(\frac{\sum_{m=0}^{H(nx)-1}W(S_m)\exp (-\sum_{m=0}^{H(nx)-1}V(S_m)), H(nx)<\infty}{E_0(\exp (-\sum_{m=0}^{H(nx)-1}V(S_m)), H(nx)<\infty)}\big)\Big) \notag\\
&=E_0\Big(\mathbb{E}\big(\frac{\sum_{m=0}^{H(nx)-1}W(S_m)\exp (-\sum_{m=0}^{H(nx)-1}V(S_m)), H(nx)<\infty}{E_0(\exp (-\sum_{m=0}^{H(nx)-1}V(S_m)), H(nx)<\infty)}\big)\Big) \notag\\
&=E_0\Big(\mathbb{E}(\sum_{m=0}^{H(nx)-1}W(S_m))\mathbb{E}\big(\frac{\exp (-\sum_{m=0}^{H(nx)-1}V(S_m)), H(nx)<\infty}{E_0(\exp (-\sum_{m=0}^{H(nx)-1}V(S_m)), H(nx)<\infty)}\big)\Big)\notag\\
&=E_0\Big(H(nx)\mathbb{E}(W(0))\mathbb{E}\big(\frac{\exp (-\sum_{m=0}^{H(nx)-1}V(S_m)), H(nx)<\infty}{E_0(\exp (-\sum_{m=0}^{H(nx)-1}V(S_m)), H(nx)<\infty)}\big)\Big) \notag\\
&=\int tdG(t)\cdot \mathbb{E}\hat{E}_{0,\omega}^{nx}(H(nx)).
\end{align}
Apply the argument as (\ref{equab 3}), we have:
\begin{equation} \label{equa 3.26}
\hat{E}_{0,\omega}^{nx}(H(nx))\leq \sum_{z'\in\mathbb{Z}^d}\hat{P}_{0,\omega}^{nx}( H(z')<H(nx))\hat{E}_{z',\omega}^{nx}(\sum_{m=0}^{H(nx)-1}\mathbf{1}_{\{S_m=z'\}}).
\end{equation}
Use (\ref{equab 4}) and the definition of $\mathcal{A}_2(0,y,(S_m)_{m\geq 0})$ given in (\ref{equab 5}):
\begin{equation}\label{equa 3.20}
\hat{E}_{0,\omega}^{nx}(H(nx))\leq D\hat{E}_{0,\omega}^{nx}(|\mathcal{A}_2(0,nx)|).
\end{equation}
Thanks to Lemma 3 in Zerner\cite{Zer98}, we have:
\begin{equation}\label{equa 3.22}
\mathbb{E}\hat{E}_{0,\omega}^{nx}(|\mathcal{A}_2(0,nx)|)\leq \frac{\big( \ln 2d+\int tdF(t)\big)}{-\ln\big(\int e^{-t}dF(t)\big)}|nx|.
\end{equation}
Substitute (\ref{equa 3.8}) in (\ref{equa 3.9}) and take the expectation:
\begin{align}\label{equa 3.21}
\frac{\mathbb{E}(a(0,nx,F*G,\omega))}{n}&\leq \frac{\mathbb{E}(a(0,nx,F,\omega))}{n}+\int tdG(t)\cdot \frac{\mathbb{E}\hat{E}_{0,\omega}^{nx}(H(nx))}{n}\notag\\
&\leq \frac{\mathbb{E}(a(0,nx,F,\omega))}{n}+\int tdG(t)\cdot D \frac{\big( \ln2d+\int tdF(t)\big)}{-\ln\big(\int e^{-t}dF(t)\big)}|x|.
\end{align}
Remark that the last inequality is from (\ref{equa 3.20}) and (\ref{equa 3.22}). Therefore,
$$\alpha_{F*G}(x)\leq \alpha_{F}(x)+c_1(F)f_1(F)\int tdG(t)|x|.$$
where $c_1(F)=D\big( \ln 2d+\int tdF(t)\big)$ and $f_1(F)=\frac{1}{-\ln\big(\int e^{-t}dF(t)\big)}$. Obviously, $c_1(F)$ and $f_1(F)$ above also satisfy  the conditions $(1)$ and $(2)$ of this Proposition. 
\end{proof}
\begin{remark} \label{remb 1}
We use the same argument as in Proposition \ref{propo 3.2} for the annealed path measure to obtain the analogous result about $\beta_{F}$. The constants chosen here are $c_1(F)=D\beta_{F}$ and $f_1(F)=\frac{1}{-\ln\big(\int e^{-t}dF(t)\big)}$.
\end{remark}
\begin{proposition}\label{propo 4.1}
Let $d\geq 3$. For any distribution function $F \in \mathcal{D}_1$ and for all $t_0>0$ such that $p:=\mathbb{P}(V(0)<t_0)<1$, we have:
\begin{equation}\label{equa 4.2}
\limsup_{|y| \to \infty}\frac{\mathbb{E}\hat{E}_{0,\omega}^{y}\Big(\sum_{m=0}^{H(y)-1}\mathbf{1}_{\{V(S_m)<t_0\}}\Big)}{|y|} \leq D(d)\bigg(\ln 2d + \int _{0}^{+\infty}tdF(t)\bigg)\cdot\frac{1}{\ln \frac{1-(1-p)e^{-t_0}}{p}} 
\end{equation}
where $D(d)$ is a constant that depends only on $d$.
\end{proposition}
\begin{proof}
We have as in (\ref{equab 3}):
\begin{align}\label{equa 3.24}
\hat{E}_{0,\omega}^{y}&\Big(\sum_{m=0}^{H(y)-1}\mathbf{1}_{\{V(S_m)<t_0\}}\Big)\notag\\
&=\sum_{z':V(z', \omega)<t_0}\hat{E}_{0,\omega}^{y}\Big(\sum_{m=0}^{H(y)-1}\mathbf{1}_{\{S_m=z'\}}\Big)\notag\\
&=\sum_{z':V(z', \omega)<t_0}\hat{P}_{0,\omega}^{y}( H(z')<H(y))\hat{E}_{z',\omega}^{y}(\sum_{m=0}^{H(y)-1}\mathbf{1}_{\{S_m=z'\}}).
\end{align}
From (\ref{equa 3.24}) and (\ref{equab 4}):
\begin{equation}\label{equa 3.19}
\hat{E}_{0,\omega}^{y}\Big(\sum_{m=0}^{H(y)-1}\mathbf{1}_{\{V(S_m)<t_0\}}\Big)\leq D\hat{E}_{0,\omega}^{y}\Big(\sum_{z\in\mathcal{A}_2(0,y)}\mathbf{1}_{\{V(z)<t_0\}}\Big).
\end{equation}
We recall that $D=\sum_{k=1}^{\infty}E_{z'}(\mathbf{1}_{\{S_k=z'\}})<\infty$ and the definition of $\mathcal{A}_2(0,y,(S_m)_{m\geq 0})$ is given by (\ref{equab 5}).\\
Take  $c=c(t_0,F):=\ln \frac{1-(1-p)e^{-t_0}}{p}$. Note that $c>0$. We use Jensen's inequality and independence as follows:
\begin{align}\label{equa 3.15}
c\mathbb{E}&\hat{E}_{0,\omega}^{y}[\sum_{z\in\mathcal{A}_2(0,y)}\mathbf{1}_{\{V(z)<t_0\}}]\notag\\
&\leq\mathbb{E}\big[\ln\hat{E}_{0,\omega}^{y}[\exp(c\sum_{z\in\mathcal{A}_2(0,y)}\mathbf{1}_{\{V(z)<t_0\}}]\big] \notag \\
&\leq\mathbb{E}\Big[a(0,y,\omega)+\notag\\
&\qquad+\ln E_0\big[\exp\big(c\sum_{z\in\mathcal{A}_2(0,y)}\mathbf{1}_{\{V(z)<t_0\}}-\sum_{z\in\mathcal{A}_2(0,y)}V(z)\big),H(y)<\infty\big]\Big] \notag \\
&\leq \mathbb{E}[a(0,y,\omega)]+\ln E_0\big[\prod_{z \in \mathcal{A}_2(0,y)}\mathbb{E}[\exp(c\mathbf{1}_{\{V(z)<t_0\}}-V(z))]\big].
\end{align}
We remark that:
\begin{align}\label{equa 3.16}
0< \mathbb{E}&[\exp(c\mathbf{1}_{\{V(z)<t_0\}}-V(z))] \notag\\
&\leq \mathbb{E}[\exp (c\mathbf{1}_{\{V(z)<t_0\}}-t_0\mathbf{1}_{\{V(z)\geq t_0\}})] \notag\\
&=\exp(-t_0)\mathbb{E}\big[\exp\big((c+t_0)\mathbf{1}_{\{V(z)<t_0\}}\big)\big] \notag\\
&=\exp(-t_0)\Big(\exp(c+t_0)\mathbb{P}(V(z)<t_0)+\mathbb{P}(V(z)\geq t_0)\Big)\notag\\
&=\exp(c)p+(1-p)\exp(-t_0)=1.
\end{align}
From (\ref{equa 3.15}) and (\ref{equa 3.16}):
\begin{equation}\label{equa 3.25}
\mathbb{E}\hat{E}_{0,\omega}^{y}[\sum_{z\in\mathcal{A}_2(0,y)}\mathbf{1}_{\{V(z)<t_0\}}]\leq \frac{\mathbb{E}[a(0,y,\omega)]}{c}.
\end{equation}
Combine this with (\ref{equa 3.19}), we obtain:
\begin{align}
\limsup_{|y| \to \infty}&\frac{\mathbb{E}\hat{E}_{0,\omega}^{y}\Big(\sum_{m=0}^{H(y)-1}\mathbf{1}_{\{V(S_m)<t_0\}}\Big)}{|y|} \notag\\
&\leq D\limsup_{|y| \to \infty}\frac{\mathbb{E}\hat{E}_{0,\omega}^{y}[\sum_{z\in\mathcal{A}_2(0,y)}\mathbf{1}_{\{V(z)<t_0\}}]}{|y|}\notag\\
&\leq \frac{D}{c}\limsup_{|y|\to\infty}\frac{\mathbb{E}(a(0,y,\omega))}{|y|}\notag\\
&\leq D\sup_{|e|=1}\alpha_{F}(e)\frac{1}{\ln \frac{1-(1-p)e^{-t_0}}{p}}\notag\\
&\leq D \Big(\ln 2d+ \int_{0}^{+\infty}tdF(t)\Big)\frac{1}{\ln \frac{1-(1-p)e^{-t_0}}{p}}.
\end{align}
It therefore implies (\ref{equa 4.2}). 
\end{proof}
\begin{remark}\label{remb 2}
It is easy to obtain an estimation of annealed path measure as in Proposition \ref{propo 4.1}:
\begin{equation}\label{equab 8 }
\limsup_{|y| \to \infty}\frac{\mathbb{\hat{E}}_{0}^{y}\Big(\sum_{m=0}^{H(y)-1}\mathbf{1}_{\{V(S_m)<t_0\}}\Big)}{|y|} \leq D(d)\bigg(\ln 2d - \ln\int_{0}^{+\infty}e^{-t}dF(t)\bigg)\cdot\frac{1}{\ln \frac{1-(1-p)e^{-t_0}}{p}}
\end{equation}
\end{remark}
We use the following corollary to check the condition (iii) of Theorem \ref{theorem 1.2}.
\begin{corollary}\label{corollary 3.1}
Let $d\geq 3$. For any $F \in \mathcal{D}_1$, $x\in\mathbb{Z}^d$, $t_0>0$, there exist $c_2(F)>0$ and $f_2(t_0,F)$ $ >0$ such that:
\begin{itemize}
\item[\rm{(1)}] $c_2(F)\leq c_2(G)$ for all distributions $F, G \in \mathcal{D}_1$ such that $F\geq G$,
\item[\rm{(2)}] $\lim_{n\to\infty}f_2(t_0,F_n)=0$ for $F_n \in \mathcal{D}_1 $ such that $\lim_{n \to \infty}F_n(t_0-)=0$,
\item[\rm{(3)}] $|\alpha_{F^{t_0}}(x)-\alpha_{F}(x)|\leq c_2(F)f_2(t_0,F)|x|$, where the definition of $F^{t_0}$ is given in (\ref{equa 1.10}).
\end{itemize}
\end{corollary}
\begin{proof}
Let $\{V(x)\}_{x\in\mathbb{Z}^d}$ be i.i.d random variables with distribution function $F$. Define:
$$W(x)=V(x)\mathbf{1}_{\{V(x)\geq t_0\}}+t_0\mathbf{1}_{\{ V(x)<t_0\}}.$$
Then $\{W(x)\}_{x\in\mathbb{Z}^d}$ are i.i.d random variables with distribution function $F^{t_0}$. 

For $x\in\mathbb{Z}^d$ and $\omega \in \Omega$, we have:
\begin{align}
a(0,& nx,F^{t_0},\omega)\notag\\
&=-\ln E_{0}[\exp(-\sum_{m=0}^{H(nx)-1}(V(S_m)\mathbf{1}_{\{V(S_m)\geq t_0\}}+t_0\mathbf{1}_{\{ V(S_m)<t_0\}})), H(nx)<\infty]\notag\\
&=-\ln\Big[\frac{E_{0}[\exp(-(\sum_{m=0}^{H(nx)-1}V(S_m)\mathbf{1}_{\{V(S_m)\geq t_0\}}+t_0\mathbf{1}_{\{ V(S_m)<t_0\}})), H(nx)<\infty]}{E_0[\exp(-\sum_{m=0}^{H(nx)-1}V(S_m)), H(nx)<\infty]}\Big]\notag\\
&\qquad +a(0,nx,F,\omega)\notag\\
&\leq -\ln\hat{E}_{0,\omega}^{nx}[\exp(-\sum_{m=0}^{H(nx)-1}t_0\mathbf{1}_{\{V(S_m)<t_0\}})]+a(0,nx,F,\omega)\notag\\
&\leq t_0\hat{E}_{0,\omega}^{nx}[\sum_{m=0}^{H(nx)-1}\mathbf{1}_{\{V(S_m)<t_0\}}]+a(0,nx,F,\omega)\notag.
\end{align}
Take the expectation, we obtain:
\begin{align}\label{equa 3.18}
\frac{\mathbb{E}a(0,nx,F^{t_0},\omega)}{n}&\leq t_0\frac{\mathbb{E}\hat{E}_{0,\omega}^{nx}[\sum_{m=0}^{H(nx)-1}\mathbf{1}_{\{V(S_m)<t_0\}}]}{n}+\frac{\mathbb{E}a(0,nx,F,\omega)}{n}\notag\\
\alpha_{F^{t_0}}(x)&\leq t_0\limsup_{n\to\infty}\frac{\mathbb{E}\hat{E}_{0,\omega}^{nx}[\sum_{m=0}^{H(nx)-1}\mathbf{1}_{\{V(S_m)<t_0\}}]}{n}+\alpha_{F}(x)\notag\\
\alpha_{F^{t_0}}(x)&\leq t_0|x|D\bigg(\ln 2d + \int _{0}^{+\infty}tdF(t)\bigg)\frac{1}{\ln \frac{1-(1-p)e^{-t_0}}{p}}+\alpha_{F}(x).
\end{align}
For the last inequality above, we applied the proposition \ref{propo 4.1}. Recall here $p=F(t_0-)$. Since $F^{t_0}\leq F$, by the monotonicity of Lyapunov exponent, we have: $\alpha_{F^{t_0}}\geq \alpha_{F}$. Combine this with (\ref{equa 3.18}), we obtain:
$$|\alpha_{F^{t_0}}(x)-\alpha_{F}(x)|\leq c_2(F)f_2(t_0,F)|x|,$$
where $c_2(F)=D\bigg(\ln 2d + \int _{0}^{+\infty}tdF(t)\bigg)$ and $f_2(t_0,F)=\frac{t_0}{\ln \frac{1-(1-p)e^{-t_0}}{p}}$. Obviously, $c_2(F)$ and $f_2(t_0,F)$ satisfy the conditions $(1)$ and $(2)$ of this Corollary.
\end{proof}
\begin{remark}\label{remb 3}
The analogous result for the annealed Lyapunov exponent holds with $c_2(F)=D\bigg(\ln 2d - \ln \int _{0}^{+\infty}e^{-t}dF(t)\bigg)$ and $f_2(t_0,F)=\frac{t_0}{\ln \frac{1-(1-p)e^{-t_0}}{p}}$.
\end{remark}
To prove the continuity of Lyapunov exponent, that is $\lim_{n\to\infty}\alpha_{F_n}(x)=\alpha_{F}(x)$ if $F_n\stackrel{w}{\to}F$, we consider first $\limsup \alpha_{F_n}(x)\leq \alpha_{F}(x)$ and then $\liminf \alpha_{F_n}(x) \geq \alpha_{F}(x)$. The inverse of a distribution function $G$ is defined in the usual way, $ G^{-1}(t)=\inf\{u\in \mathbb{R}: G(u)>t\}, t\in \mathbb{R}$. We will need the following lemma from \cite{Cox81}.
\begin{lemma}\label{lemme 4.1} 
If $(F_n)\in \mathcal{D}$ such that $F_n\leq F$ and $F_n \stackrel{w}{\to}F$, then $F_n^{-1}\to F^{-1}$ pointwise on $[0, 1)$.
\end{lemma}
The proof of the following theorem is analogous to the proof of theorem 1.13 in \cite{Cox81}.
\begin{theorem}\label{theorem 4.1}
Let $d\geq 1$. Let $(F_n) \in \mathcal{D}$ such that $F_n\geq G$ for all $n$ and for some $G\in \mathcal{D}_1$. If $ F_n\stackrel{w}{\to}F $, then $\limsup_{n\to\infty}\alpha_{F_n}(x)\leq \alpha_{F}(x)$ for all $x\in\mathbb{Z}^d$.
\end{theorem}
\begin{proof}
As mentioned in the Introduction after theorem \ref{theorem 1.2}, there are only two cases to consider. If $F_n \geq F$, by the monotonicity property, $\alpha_{F_n}(x)\leq \alpha_{F}(x)$ for all $n$, hence $\limsup \alpha_{F_n}(x) \leq \alpha_{F}(x)$. For the rest of the proof we shall assume $F_n\leq F$. Let $\xi(x)$, $x\in\mathbb{Z}^d$ be an i.i.d family of uniform random variables on $(0,1)$. Let $V(x):=F^{-1}(\xi(x))$, $V_n(x):=F^{-1}_n(\xi(x))$ and $W(x)=G^{-1}(\xi(x))$. Then $V(x), V_n(x)$ and $W(x)$ are i.i.d. families of, respectively, $F$-distributed, $F_n$-distributed et $G$-distributed random variables. Furthermore, for each $x\in\mathbb{Z}^d$:
$$ V(x)\leq V_n(x)\leq W(x) \mbox{ \hspace{0.1cm} a.s}.  $$ 
and by lemma \ref{lemme 4.1}:
$$ \lim_{n\to\infty}V_n(x)=V(x) \mbox{ \hspace{0.1cm} a.s}.  $$
For $x\in\mathbb{Z}^d$ and for $k\in\mathbb{N}$, we have:
$$ a(0, kx,F, \omega)=-\ln E_0(\exp(-\sum_{m=0}^{H(kx)-1}V(S_m)), H(kx)<\infty). $$
$$ a(0, kx, F_n,\omega)=-\ln E_0(\exp(-\sum_{m=0}^{H(kx)-1}V_n(S_m)), H(kx)<\infty). $$
For a fixed path of the simple random walk, we have:
$$ \exp(-\sum_{m=0}^{H(kx)-1}V_n(S_m))\mathbf{1}_{\{H(kx)<\infty\}} \to \exp(-\sum_{m=0}^{H(kx)-1}V(S_m))\mathbf{1}_{\{H(kx)<\infty\}}, $$
when $n \to \infty $. Furthermore, 
$$ \exp(-\sum_{m=0}^{H(kx)-1}V_n(S_m))\mathbf{1}_{\{H(kx)<\infty\}}<1, $$
for all $n$. By the dominated convergence theorem: $a(0, kx, F_n, \omega) \to a(0, kx, F, \omega)$ when $n\to\infty$. With $k$ and $x$ fixed, we have:
$$a(0, kx, F_n,\omega)\leq a(0, kx, G,\omega), $$
for all $n$ since $V_n(z)\leq W(z)$ for all $n$ and for all $z\in\mathbb{Z}^d$. Moreover:
$$ \mathbb{E}(a(0,kx,G,\omega))<k|x|(\ln 2d+\int t dG(t))<\infty. $$
Apply again the dominated convergence theorem:
$$ \mathbb{E}a(0, kx,F_n,\omega)\stackrel{n\to\infty}{\to}\mathbb{E}a(0, kx, F,\omega). $$
Now fix $\epsilon>0$. Choose $K$ large enough such that:
$$ 0\leq \frac{\mathbb{E}a(0, Kx, F,\omega)}{K}-\alpha_{F}(x)<\epsilon .$$
Now choose $N$ such that for all $n\geq N$:
$$ 0\leq \frac{\mathbb{E}a(0, Kx, F_n,\omega)}{K}-\frac{\mathbb{E}a(0, Kx, F, \omega)}{K}<\epsilon .$$
We have hence for all $n\geq N$:
\begin{align}
0 &\leq \alpha_{F_n}(x)-\alpha_{F}(x) \mbox{ ( Since $F_n\leq F$)} \notag \\
&\leq \frac{\mathbb{E}a(0, Kx, F_n,\omega)}{K}-\alpha_{F}(x) \mbox{ ( Since $\alpha_{F_n}(x)=\inf_{k\geq 1}\frac{\mathbb{E}a(0, kx, F_n,\omega)}{k}$)} \notag \\
& \leq \frac{\mathbb{E}a(0,Kx,F_n,\omega)}{K}- \frac{\mathbb{E}a(0,Kx,F,\omega)}{K}+\frac{\mathbb{E}a(0,Kx,F,\omega)}{K}-\alpha_{F}(x) \notag\\
& \leq \epsilon+\epsilon \notag.
\end{align}
Therefore, $\limsup \alpha_{F_n}(x)\leq \alpha_{F}(x)+2\epsilon $. Now let $\epsilon \to 0$, we have the result.
\end{proof}
\begin{theorem}\label{theoremb 2}
Let $d\geq 1$ and $(F_n)\in \mathcal{D}$ such that $F_n\stackrel{w}{\to}F$, then $\limsup_{n\to\infty}\beta_{F_n}(x)\leq \beta_{F}(x) $ for all $x\in\mathbb{Z}^d$.
\end{theorem}
The proof of this theorem is very similar to the theorem \ref{theorem 4.1}. We only want to remark that $b(0,kx,F_n)$ always converge to $b(0,kx,F)$ with $k$ and $x$ fixed when $n \to \infty$ (without the condition of finite mean). \\
\textbf{Proof of theorem \ref{theorem 1.3}:}\\
Case $d=1$. Thanks to Proposition 10 in Zerner\cite{Zer98} and the proof of theorem \ref{theorem 4.1}, we have for all $m\in\mathbb{N}$: 
$$\alpha_{F}(m)=\mathbb{E}(a(0,m,F))=\lim_{n\to\infty}\mathbb{E}(a(0,m,F_n))=\lim_{n\to\infty}\alpha_{F_n}(m),$$

Case $d\geq 3$. By the monotonicity property of Lyapunov exponent, proposition \ref{propo 3.2} and corollary \ref{corollary 3.1}, we see that all the conditions of theorem \ref{theorem 1.2} in the Introduction are satisfied, we have then $\liminf_{n \to \infty}\alpha_{F_n}(x)\geq \alpha_{F}(x)$. Combine this with  $\limsup_{n\to\infty}\alpha_{F_n}(x)\leq \alpha_{F}(x)$ given by theorem \ref{theorem 4.1} to obtain finally theorem \ref{theorem 1.3}.

Case $d=2$. First, we note that if there exists $\lambda>0$ such that $V\geq \lambda$, then $\hat{E}_{z',\omega}^{y}(\sum_{m=0}^{H(y)-1}\mathbf{1}_{\{S_m=z'\}})$ $ \leq D(\lambda)<\infty$ for all $z' \in \mathbb{Z}^d$. Indeed, from (\ref{equab 4}), we have:
\begin{align}\label{equab 10}
\hat{E}_{z',\omega}^{y}(\sum_{m=0}^{H(y)-1}\mathbf{1}_{\{S_m=z'\}})&\leq 1+ \sum_{k=1}^{+\infty}E_{z'}\big(\mathbf{1}_{\{S_k=z'\}} \exp(-\sum_{m=0}^{k-1}V(S_m))\big)\notag\\
&\leq \sum_{k=0}^{+\infty}\exp(-k \lambda)=\frac{1}{1-\exp(-\lambda)} <\infty
\end{align}
With (\ref{equab 10}), we can now repeat all argument used in the case $d\geq 3$ to show that  Proposition \ref{propo 3.2}, Corollary \ref{corollary 3.1} and then Theorem \ref{theorem 1.3} also hold when $d=2$.

In first passage percolation, \cite{Wier78} truncates the distribution function below and above at $t_0>0$ and shows the continuity of time constant in these two cases. By the theorem \ref{theorem 1.3}, we obtain the similar results for Lyapunov exponent as its corollaries.
\begin{corollary}
Let $F^{t_0}$ be the distribution function obtained by truncating $F$ below at $t_0>0$(see (\ref{equa 1.10}) for the definition). Then, if $F \in \mathcal{D}_1$, for all $x\in\mathbb{Z}^d$ :
\begin{equation}
\lim_{t_0\rightarrow 0}\alpha_{F^{t_0}}(x)=\alpha_{F}(x).
\end{equation}
\end{corollary}
\begin{corollary}
Let $^{t_0}F$ be the distribution function obtained by truncating above at $t_0>0$, i.e:
\begin{equation}\label{equa 4.3}
^{t_0}F(t):=
\begin{cases}
F(t) \mbox{ if $t<t_0$}\\
1 \mbox{ if $t\geq t_0$},
\end{cases}
\end{equation}
Then, if $F \in \mathcal{D}_1$, for all $x\in \mathbb{Z}^{d}$:
\begin{equation}
\lim_{t_0\to\infty}\alpha_{^{t_0}F}(x)=\alpha_{F}(x).
\end{equation}
\end{corollary}
We now consider the annealed Lyapunov exponent. The following theorem parallels lemma 2 of \cite{Cox-Kes81} and theorem 7.12 of \cite{Wier78} in the context of first passage percolation. This is a tool to eliminate the condition of finite mean.  As in \cite{Flury06}, we define for $z\in\mathbb{Z}^d, n \in\mathbb{N}$ the number of visits to the site $z$ by the random walk up to time $n$:
$$\ell_z(n):=|\{m\in\mathbb{N}_0: m<n, S_m=z\}|.$$
This notation is useful in the proof of theorem \ref{theoremb 3}.
\begin{theorem}\label{theoremb 3} 
Let $d\geq 3$. Let $F\in \mathcal{D}$ such that $F$ assigns probability $1$ to $[0,+\infty[$. Then for all $x\in \mathbb{Z}^d$:
\begin{equation}\label{equation 16}
\lim_{t_0\to\infty}\beta_{^{t_0}F}(x)=\beta_{F}(x),
\end{equation}
where $^{t_0}F$ is defined in (\ref{equa 4.3}).
\end{theorem}
\begin{proof}
Let $V_1(x)$ and $V_2(x)$, $x \in \mathbb{Z}^{d}$ be two families of i.i.d random potentials  with distribution $F$, independent of one another. Then, $W_{t_0}(x):=\min\{V_1(x);t_0\}$, $x\in\mathbb{Z}^d$ are i.i.d random potentials  with distribution function $^{t_0}F$.
Define a distribution function $^{t_0}\hat{F}$ by:
\begin{equation} \label{equa 3.13}
^{t_0}\hat{F}:=
\begin{cases}
0 \mbox{ if $t<0$}\\
F(t_0) \mbox{ if $0\leq t \leq t_0$}\\
F(t) \mbox{ if $t>t_0$}
\end{cases}
\end{equation}
Take $U_{t_0}(x)=V_2(x)\mathbf{1}_{\{V_2(x)>t_0\}}$. Hence $U_{t_0}(x), x\in \mathbb{Z}^d$ is an i.i.d family of random potentials with distribution function $^{t_0}\hat{F}$. $(W_{t_0}+U_{t_0})(x), x\in \mathbb{Z}^d$ are i.i.d random potentials with distribution function $^{t_0}F*^{t_0}\hat{F}$. Moreover,
\begin{align}\label{equation 17}
b&(0,nx,^{t_0}F*^{t_0}\hat{F})\notag\\
&=-\ln \mathbb{E} E_0[\exp (-\sum_{m=0}^{H(nx)-1}W_{t_0}(S_m)-\sum_{m=0}^{H(nx)-1}U_{t_0}(S_m)), H(nx)<\infty] \notag \\
&=-\ln \Big[\frac{ \mathbb{E} E_0\big[\exp (-\sum_{m=0}^{H(nx)-1}W_{t_0}(S_m)-\sum_{m=0}^{H(nx)-1}U_{t_0}(S_m)), H(nx)<\infty\big]}{\mathbb{E}e(0,nx,\omega,^{t_0}F)}\notag\\
&\hspace{2cm} \cdot \mathbb{E}e(0,nx,\omega,^{t_0}F) \Big]\notag \\
&=-\ln \frac{ \mathbb{E} E_0\big[\exp (-\sum_{m=0}^{H(nx)-1}W_{t_0}(S_m)-\sum_{m=0}^{H(nx)-1}U_{t_0}(S_m)), H(nx)<\infty\big]}{\mathbb{E}e(0,nx,\omega,^{t_0}F)}\notag\\
&\hspace{2cm}+b(0,nx,^{t_0}F).
\end{align}
Since two sequences $(W_{t_0}(x))_{x\in\mathbb{Z}^d}$ and $(U_{t_0}(x))_{x\in\mathbb{Z}^d}$ independent of each other, the first term in right hand side of (\ref{equation 17}) is equal to:
\begin{align}\label{equation 18}
&-\ln\frac{ E_{0}\Big(\mathbb{E}[\exp(-\sum_{m=0}^{H(nx)-1}U_{t_0}(S_m))]\mathbb{E}[\exp (-\sum_{m=0}^{H(nx)-1}W_{t_0}(S_m))], H(nx)<\infty\Big)}{\mathbb{E}e(0,nx,\omega,^{t_0}F)}\notag\\
&=-\ln\frac{ E_{0}\Big(\mathbb{E}[\exp\big(-\sum_{z\in\mathbb{Z}^d}\ell_z(H(nx))U_{t_0}(z)\big)]\mathbb{E}[\exp (-\sum_{m=0}^{H(nx)-1}W_{t_0}(S_m))], H(nx)<\infty\Big)}{\mathbb{E}e(0,nx,\omega,^{t_0}F)}\notag\\
&=-\ln\frac{ E_{0}\Big(\prod_{z\in\mathbb{Z}^d}\mathbb{E}[\exp\big(-\ell_z(H(nx))U_{t_0}(z)\big)]\mathbb{E}[\exp (-\sum_{m=0}^{H(nx)-1}W_{t_0}(S_m))], H(nx)<\infty\Big)}{\mathbb{E}e(0,nx,\omega,^{t_0}F)}\notag\\
&\leq -\ln\frac{ E_{0}\Big(\prod_{z\in\mathbb{Z}^d}[\mathbb{E}\exp(-U_{t_0}(z))]^{\ell_z(H(nx))}\mathbb{E}[\exp (-\sum_{m=0}^{H(nx)-1}W_{t_0}(S_m))], H(nx)<\infty\Big)}{\mathbb{E}e(0,nx,\omega,^{t_0}F)}\notag\\ 
&= -\ln\frac{ E_{0}\Big([\mathbb{E}\exp(-U_{t_0}(0))]^{H(nx)}\mathbb{E}[\exp (-\sum_{m=0}^{H(nx)-1}W_{t_0}(S_m))], H(nx)<\infty\Big)}{\mathbb{E}e(0,nx,\omega,^{t_0}F)}\notag\\ 
&=-\ln \mathbb{\hat{E}}_{0,^{t_0}F}^{nx}(\mathbb{E}(\exp(-U_{t_0}(0)))^{H(nx)})\leq (-\ln \mathbb{E}(\exp(-U_{t_0}(0))))\mathbb{\hat{E}}_{0,^{t_0}F}^{nx}(H(nx)).
\end{align}
For the first inequality of (\ref{equation 18}), we remark that for all $z\in\mathbb{Z}^d$: 
$$\mathbb{E}(\exp(-\ell_z(H(nx)U_{t_0}(z))))\geq [\mathbb{E}\exp(-U_{t_0}(z))]^{\ell_z(H(nx))}.$$ 
This inequality is obvious if $\ell_z(H(nx))=0 $, and follows from Jensen's inequality if \mbox{$\ell_z(H(nx))\geq 1$}.\\
From (\ref{equab 11}) (in the proof of part (ii) of theorem \ref{theorem 1.1}):
\begin{equation}\label{equation 20}
\mathbb{\hat{E}}_{0,^{t_0}F}^{nx}(H(nx))\leq \frac{D}{-\ln \mathbb{E}\exp(-W_{t_0}(0))}b(0,nx,^{t_0}F).
\end{equation}
From (\ref{equation 17}), (\ref{equation 18}) and (\ref{equation 20}), for all $x\in \mathbb{Z}^d$:
\begin{align}\label{equation 19}
\frac{b(0,nx,^{t_0}F*^{t_0}\hat{F})}{n}&\leq (-\ln \mathbb{E}(\exp(-U_{t_0}(0))))\frac{D}{-\ln \mathbb{E}\exp(-W_{t_0}(0))}\frac{b(0,nx,^{t_0}F)}{n} \notag\\
&\hspace{5cm}+ \frac{b(0,nx,^{t_0}F)}{n}\notag\\ 
\beta_{^{t_0}F*^{t_0}\hat{F}}(x) &\leq (-\ln \mathbb{E}(\exp(-U_{t_0}(0))))\frac{D}{-\ln \mathbb{E}\exp(-W_{t_0}(0))} \beta_{^{t_0}F}(x)+\beta_{^{t_0}F}(x)
\end{align}
Note that $\lim_{t_0 \to \infty}(-\ln \mathbb{E}(\exp(-U_{t_0}(0))))=0$  and $\lim_{t_0 \to \infty}-\ln \mathbb{E}\exp(-W_{t_0}(0))=-\ln \mathbb{E}\exp(-V_1(0)$ $ =const$. 
It is clear that $^{t_0}F\geq F \geq ^{t_0}F*^{t_0}\hat{F}$. Then, $\beta_{^{t_0}F}(x)\leq \beta_{F}(x)\leq \beta_{^{t_0}F*^{t_0}\hat{F}}(x)$. Combine this with (\ref{equation 19}):
\begin{align}
\limsup_{t_0 \to \infty}\beta_{^{t_0}F}(x)\leq \beta_{F}(x)&\leq \lim_{t_0 \to \infty}(-\ln \mathbb{E}(\exp(-U_{t_0}(0))))\frac{D}{-\ln \mathbb{E}\exp(-W_{t_0}(0))}\beta_{F}(x)+\liminf_{t_0 \to \infty}\beta_{^{t_0}F}(x) \notag\\
&\leq \liminf_{t_0 \to \infty}\beta_{^{t_0}F}(x).  
\end{align}
We get finally theorem \ref{theoremb 3}.
\end{proof}
\begin{remark}\label{remb 4}
In the case of $d=2$, beside the hypothesis of Theorem \ref{theoremb 3}, we further assume that there exists $\lambda >0$ such that $F(\lambda)=0$, then Theorem \ref{theoremb 3} also holds. The argument we use here is as the proof of the case $d\geq 3$ with remark that:
$$ \mathbb{\hat{E}}_{0,^{t_0}F}^{nx}(H(nx))\leq \frac{1}{1-\exp(-\lambda)}\cdot \frac{1}{-\ln \mathbb{E}\exp(-W_{t_0}(0))}b(0,nx,^{t_0}F), $$
\end{remark}
We are now ready to prove Theorem \ref{theoremb 1}.\\
\textbf{Proof of theorem \ref{theoremb 1}.}

Case $d=1$. It is similar to the case quenched Lyapunov exponent.

Case $d \geq 3$. First, remark that by the same arguments we used to prove Theorem \ref{theorem 1.3}, we also have that: let $(F_n)$ is a sequence of distribution functions such that there is a distribution function $G$ with finite mean such that $G \leq F_n$ for all $n$. If $F_n\stackrel{w}{\to}F$, then $\lim_{n\to\infty}\beta_{F_n}(x)=\beta_{F}(x) $ for all $x\in\mathbb{Z}^d$.\\
Now, fix $t_0>0$ which is a continuity point of $F$. Define a distribution function $G$:
\begin{equation}
G(t):=
\begin{cases}
0 \mbox{ if $t<t_0$}\\
1 \mbox{ if $t\geq t_0$},
\end{cases}
\end{equation}
Clearly $\int t dG(t)=t_0$ and $G(t)\leq {^{t_0}F_n(t)}$ for all $n$. Using the result above to have $\lim_{n \to \infty}\beta_{^{t_0}F_n}=\beta_{^{t_0}F}$. Furthermore $\liminf_{n\to\infty}\beta_{F_n}\geq \liminf_{n \to \infty}\beta_{^{t_0}F_n}=\beta_{^{t_0}F}$ since  $^{t_0}F_n\ge F_n$ for all $n$. Now let $t_0 \to \infty$ through continuity points of $F$, and apply Theorem \ref{theoremb 3}, $\liminf_{n\to\infty}\beta_{F_n}\geq \lim_{t_0 \to \infty}\beta_{^{t_0}F}=\beta_{F}$. Combine this with $\limsup_{n\to \infty}\beta_{F_n}\leq \beta_{F}$ given by Theorem \ref{theoremb 2}  to obtain finally Theorem \ref{theoremb 1}.

Case $d=2$. If we further suppose that there exists $\lambda>0$ such that $F_n(\lambda)=0$ for all $n$, we can use Remark \ref{remb 4} and follow the arguments as the case $d\geq 3$ to obtain Theorem \ref{theoremb 1}.

\begin{acknowledgements}
I would like to thank my Ph.D. advisor Daniel Boivin for many helpful discussions and suggestions about this work. This research was supported by the French ANR project MEMEMO2, 2010 BLAN 0125 04.
\end{acknowledgements} 
\bibliographystyle{plain}
\bibliography{filebib}

\end{document}